\documentclass{amsart}

\usepackage{amssymb}
\usepackage{mathrsfs}
\usepackage{amsmath}
\usepackage{amscd}
\usepackage{graphicx}
\usepackage{pinlabel}

\usepackage{cite}

\newtheorem{theorem}{Theorem}
\newtheorem{prop}[theorem]{Proposition}
\newtheorem{lemma}[theorem]{Lemma}
\newtheorem{corollary}[theorem]{Corollary}
\newtheorem{definition}[theorem]{Definition}
\newtheorem{rmk}[theorem]{Remark}

\begin{document}
\title[polyhedral metrics on triangulated surfaces]
{A characterization of the polyhedral metrics on triangulated surfaces}
\renewcommand{\theequation}{\arabic{section}.\arabic{subsection}.\arabic{equation}}
\numberwithin{equation}{section}
\numberwithin{theorem}{section}

\author{Tianqi Wu}
\maketitle

\begin{abstract}
Given a triangulated surface, a polyhedral metric could be constructed by gluing Euclidean triangles edge-to-edge. We carefully describe the construction and prove that such a polyhedral metric is the only intrinsic metric on the glued surface that preserves the lengths of the curves on the Euclidean triangles. We also discuss the edge length coordinates of the Teichm\"uller space of the polyhedral metrics on a marked surface.
\end{abstract}


\section{Introduction}
Polyhedral metrics and (geodesic) triangulations on surfaces have been widely studied and relied on in mathematical theories and computational geometry.
Given a triangulation of a surface, a polyhedral metric could be constructed by gluing Euclidean triangles edge-to-edge. Such a construction seems obvious.
In this paper,
we carefully describe the construction in details, and rigorously justify fundamental properties of such a constructed polyhedral metrics. Specifically, we defined the ``glued metric'' based on the idea of quotient metric, and proved that arc lengths are preserved after gluing. We also discuss the edge length coordinates of the Teichm\"uller space of polyhedral metrics, briefly explaning the connection with the Teichm\"uller theory of hyperbolic surfaces with cusps and the theory of discrete conformal geometry.

In Section 2, we give a detailed specific construction of $\Delta$-complexes by gluing Euclidean triangles. 
Then in Section 3, we construct polyhedral metrics on triangulated surfaces and justify that the constructed one is indeed a metric. 
In Section 4, we proved the arc lengths preserving property.
In section 5, we define the Teichm\"uller space $T_{PL}$ of polyhedral metrics, and introduced a $\mathbb R^n$-parameterization with the help of the Teichm\"uller space of hyperbolic surfaces with cusps. We also discuss the edge lengths coordianates and the connection with discrete conformal geoemetry.

\section{Triangulations of a closed surface}
Denote

(a) $D$ as a fixed closed regular triangle in $\mathbb R^2$ with the edge length 1,

(b) $V_0$ as the set of the three vertices of $D$,

(c) $E_0$ as the set of the three closed edges of $D$,

(d) $D^o$ as the interior of $D$ in $\mathbb R^2$, and

(e) $e^o=e-V_0$ as an open edge for all $e\in E_0$.
\\
\begin{definition}
\label{definition of a triangulation}
    A \emph{triangulation} $T$ of a connected closed surface $S$ is represented by a finite index set $I=\{1,2,...,n\}$ for some $n\in\mathbb Z_{>0}$ and continuous maps $\sigma_i:D\rightarrow S$ for $i\in I$
such that

(a) $\sigma_i|_{e^o}$ is injective for all $i\in I$ and $e\in E_0$,

(b) $\sigma_i|_{D^o}$ is injective for all $i\in I$,

(c) $S$ is a disjoint union of $\cup_{i}\sigma_i(V_0)$, $\sigma_i(D^o)$ for $i\in I$, and elements in 
$$
\{\sigma_i(e^o):e\in E_0, ~i\in I\},
$$ 
and

(d) if $\sigma_{i_1}(e_1^o)=\sigma_{i_2}(e_2^o)$ for some $i_1,i_2\in I$ and $e_1, e_2\in E_0$, then there exists a linear homeomorphism (i.e., an isometry) $h:e_1\rightarrow e_2$ such that
$$
\sigma_{i_2}(h(x))=\sigma_{i_1}(x)
$$
for all $x\in e_1$.

We denote such a triangulation as $T=(\sigma_i)_{i\in I}$.
\end{definition}

\begin{rmk}
    Notice that by part (c) of Definition 2,1, 
    for all $(e_1,i_1),(e_2,i_2)\in E_0\times I$, $\sigma_{i_1}(e_1^o)=\sigma_{i_2}(e_2^o)$ if
$\sigma_{i_1}(e_1^o)\cap\sigma_{i_2}(e_2^o)\neq\emptyset$.
\end{rmk}

\begin{rmk}
Such a  triangulation $T$ is naturally a $\Delta$-complex, but not necessary a simplicial complex. Two edges of the same triangle could possibly be glued together, and two triangles may be glued along more than one pairs of edges. As a consequence, such a triangulation $T$ is naturally a CW-complex but may not be regular.    
\end{rmk}

Given a triangulation $T=(\sigma_i)_{i\in I}$ of a closed surface $S$,

(a) the maps $\sigma_i$ for $i\in I$ are called the \emph{characteristic maps} of $T$, 

(b) $\sigma_i(v)$ is called a \emph{vertex} of $T$ if $v\in V_0$ and $i\in I$,

(c) $V(T)=\cup_{i\in I}\sigma_i(V_0)$ denotes the set of vertices of $T$,

(d) $\sigma_i(e)$ is called an \emph{edge} of $T$ if $e\in E_0$ and $i\in I$, 

(e) $E=E(T)$ denotes the set of edges of $T$,

\section{Construction of polyhedral metrics}
In this section we will construct polyhedral metrics on $S$, with a given triangulation and prescribed edge lengths.

Given $l\in\mathbb R^{E_0}_{>0}$ satisfying the triangle inequality, there exists a linear map $f:\mathbb R^2\rightarrow\mathbb R^2$ such that $f(D)$ is a Euclidean triangle where the edge length of $f(e)$ is $l_e$ for all $e\in E_0$. 
Let $d[l]$ be the induced metric on $D$, defined by
$$
d[l](x,y)=|f(x)-f(y)|_2.
$$
Notice that the metric $d[l]$ is independent on the choice of $f$.
Suppose $T=(\sigma_i)_{i\in I}$ is a triangulation of $S$, denote $\mathbb R^{E(T)}_\Delta$ as a real-valued function on $E(T)$ satisfying the triangle inequalities. Specifically,
$$
\mathbb R^{E(T)}_\Delta=\{l\in\mathbb R_{>0}^{E(T)}:
l_{\sigma_i(e_1)}+l_{\sigma_i(e_2)}>l_{\sigma_i(e_3)}\text{ if $\{e_1,e_2,e_3\}=E_0$ and $i\in I$}\}.
$$
Suppose $l\in\mathbb R^{E(T)}_\Delta$ and $l^{(i)}\in\mathbb R^{E_0}_{>0}$ denotes the restricted edge length on $\sigma_i(D)$, i.e.,
$$
l^{(i)}_{e}=l_{\sigma_i(e)}
$$
for all $e\in E_0$.

A potential metric on $\sigma_i(D)$ could be
$$
d_i(x,y)=\min\{d[l^{(i)}](x',y'):x'\in\sigma_i^{-1}(x),~y'\in\sigma_i^{-1}(y)\}
$$
for all $x,y\in \sigma_i(D)$. Notice that in general $\sigma_i$ is not injective on $D$ and $\sigma_i^{-1}(x)$ may contain more than one point. Such $d_i$ actually may not be a metric on $\sigma_i(D)$. This is fine and we can still construct a metric $d[l]$ on $S$ as
$$
d[l](x,y)=\inf
\bigg\{
\sum_{j=1}^kd_{i_j}(x_{j-1},x_j):x_{j-1},x_j\in\sigma_{i_j}(D)\text{ for some $i_j\in I$},~
x=x_0,~y=x_k
\bigg\}
$$
for all $x,y\in S$.
To justify that $d[l]$ is indeed a metric on $S$, we need to prove that for all $x,y,z\in S$,
 
(a) $d[l](x,y)<\infty$ ,

(b) $d[l](x,y)\geq0$,

(c) 
$d[l](x,y)=0$ only if $x=y$,

(d) 
$$
d[l](x,y)=d[l](y,x),
$$ 
and

(e) 
$$
d[l](x,z)\leq d[l](x,z)+d[l](y,z).
$$
Here part (a) is indeed Proposition 3.1 below. Part (b) is obvious from the definition. Part (c) is an immediate consequence of Proposition 3.4 below. It is straightforward to verify part (d) and (e) from the definition of $d[l]$.

Furthermore, $d[l]$ is compatible with the quotient topology on $S$ (see Proposition 3.5). 

\begin{prop}
For all $x,y\in S$,
$d[l](x,y)<\infty$, i.e.,
there exists 
$$
x_0,x_1,...,x_k\in S
$$ 
such that  $x_0=x$, $x_k=y$, and for all $j=1,...,k$, $x_{j-1},x_j\in \sigma_{i_j}(D)$ for some $i_j\in I$.
\end{prop}
\begin{proof}
Given $x\in S$, denote
$$
C(x)=\{y\in  S:d[l](x,y)<\infty\}.
$$
It suffices to show that $C(x)$ is the whole space $S$. Since $S$ is connected and $x\in C(x)$, we only need to show that $C(x)$ is both open and closed.
It is easy to see that $\sigma_i(D)\subset C(x)$ if $\sigma_i(D)\cap C(x)\neq\emptyset$.
Let 
$$
J=\bigg\{i\in I:~\sigma_i(D)\cap C(x)\neq\emptyset\bigg\}.
$$
and then
$$
C(x)=\bigcup_{i\in J}\sigma_i(D)
$$ 
is compact and closed.
For all $i\in I-J$, 
$$
C(x)\cap \sigma_i(D)=\emptyset.
$$
So 
$$
\bigcup_{i\in I-J}\sigma_i(D)
$$ 
is the complement of $C(x)$ in $S$ and is compact. So $C(x)$ is open.

\end{proof}

For all $x\in D$,
denote 
$$
K(x)=\cup\{e\in E_0: x\notin e\}
$$
as the union of the edges not containing $x$. In particular, 

(a) if $x\in D^o$, then $K$ is the union of all the three edges of $D$,

(b) if $x\in e^o$ for some $e\in E_0$, then $K$ is the union of the other two edges of $D$, and

(c) if $x\in V_0$, then $K$ is the edge opposite to $x$ in $D$. 
\\
For all $x\in D$ and $i\in I$,
denote 
$r_i(x)>0$ as the distance between $K(x)$ and $x$ in $(D,d[l^{(i)}])$.
For all $x\in S$,
let
$$
r(x)=
\min\bigg\{r_i(x'):i\in I, ~x'\in \sigma_i^{-1}(x)\bigg\}.
$$
\begin{lemma}
If $x,z\in\sigma_i(D)$, $d_i(x,z)<r(x)$, and $z'\in\sigma_j^{-1}(z)$,
then there exists $x'\in \sigma_j^{-1}(x)$ such that
$$
d[l^{(j)}](x',z')=d_i(x,z).
$$ 
\end{lemma}
\begin{proof}
Assume
$$
d_i(x,z)=d[l^{(i)}](x'',z'')<r(x)
$$
for some $x''\in\sigma_i^{-1}(x)$, $z''\in\sigma_i^{-1}(z)$.
The conclusion is trivial if $x=z$ or $(z'',i)=(z',j)$.
So we may assume that $x\neq z$ and $(z'',i)\neq(z',j)$. 
By the definition of $r(x)$,
$$
z''\notin  K(x'')\cup\{x''\}\supset V_0.
$$
Then by $(z',i)\neq(z'',j)$ and part (c) of the definition of a triangulation, there exists $e,e'\in E_0$ such that
$$
z''\in e^o,
$$ 
$$
z'\in e',
$$ 
and
$$
\sigma_i(e)=\sigma_j(e').
$$
Furthermore, $e\not\subset K(x'')$ and thus $x''\in e$.

By part (d) of the definition of a triangulation, there exists a linear homeomorphism $h:e\rightarrow e'$ such that $\sigma_j(h(w))=\sigma_i(w)$ for all $w\in e$. 
Such $h$ is an isometry from $(e,d[l^{(i)}])$ to $(e',d[l^{(j)}])$ since

(a) $l_{\sigma_i(e)}=l_{\sigma_{j}(e')}$,

(b) the metrics $d[l^{(i)}],d[l^{(j)}]$ are induced from linear maps to Euclidean triangles in $\mathbb R^2$, and

(c) the map $h$ is linear.
\\
Then $h(z'')=z'$, $h(x'')\in\sigma_j^{-1}(x)$, and
$$
d[l^{(j)}]( h(x''),z')=d[l^{(j)}]( h(x''),h(z''))=d[l^{(i)}](x'',z'')=d_i(x,z).
$$
\end{proof}

\begin{corollary}
If $x,z\in\sigma_i(D)$, $d_i(x,z)<r(x)$, and $z\in\sigma_j(D)$,
then 
$$
x\in\sigma_j(D)
$$ 
and
$$
d_i(x,z)=d_j(x,z).
$$     
\end{corollary}
\begin{proof}
 Let $z'\in\sigma_j^{-1}(z)$. By Lemma 3.2, there exists $x'\in\sigma_j^{-1}(x)$ such that
 $$
d_j(x,z)\leq d[l^{(j)}](x',z')=d_i(x,z)<r.
 $$
 By the same reason
 $$
d_i(x,z)\leq d_j(x,z) 
  $$
  and thus
  $$
  d_i(x,z)=d_j(x,z).
  $$
\end{proof}

\begin{prop}
Given $x,y\in S$,
$$
d[l](x,y)\geq r(x)
$$
or
$$
d[l](x,y)=d_i(x,y)
$$
for some $i\in I$ with $x,y\in\sigma_i(D)$.
\end{prop}
\begin{proof}
Let us prove by contradiction and assume that this is not true. 
By the definition of $d[l]$ there exists $y\in S$ and $x_0,...,x_k$ such that

(a) for all $j\in\{1,...,k\}$, $x_{j-1},x_j\in D_{i_j}$ for some $i_j\in I$,

(b) $x_0=x$ and $x_k=y$,

(c) 
$$
\sum_{j=1}^k d_{i_j}(x_{j-1},x_j)<r,
$$
and

(d) 
$$
\sum_{j=1}^k d_{i_j}(x_{j-1},x_j)<d_i(x,y)
$$
if $i\in I$ and $x,y\in \sigma_i(D)$.

Now let $m\geq1$ be the maximum integer in $\{1,...,k\}$ such that 
for all $j=1,...,m$, 

(a) $x\in\sigma_{i_j}(D)$, and

(b) 
$
d_{i_{j}}(x,x_{j-1})<r(x).
$
\\
If $m=k$, then $x,y\in\sigma_{i_m}(D)$ and
$$
d_{i_{m}}(x,x_{m})=d_{i_{m}}(x,y).
$$
\\
If $k'<k$, then
$$
d_{i_{m}}(x,x_{m})\geq r(x),
$$
since otherwise by Corollary 3.3 we have that 
$$
x\in \sigma_{i_{m+1}}(D)
$$ 
and
$$
d_{i_{m+1}}(x,x_{m})<r(x)
$$
contradicting with the maximality of $m$.

So it suffices to show that
$$
\sum_{j=1}^{m} d_{i_j}(x_{j-1},x_j)\geq d_{i_{m}}(x,x_{m}).
$$
Since 
$$
d_{i_1}(x_0,x_1)=d_{i_1}(x,x_1),
$$
it suffices to show that for all $j=2,...,m$,
$$
d_{i_j}(x_{j-1},x_j)\geq d_{i_{j}}(x,x_{j})-d_{i_{j-1}}(x,x_{j-1}),
$$
which by Corollary 3.3 is equivalent to
$$
d_{i_{j}}(x,x_{j-1})+d_{i_j}(x_{j-1},x_j)\geq d_{i_{j}}(x,x_{j}).
$$
Here we prove a general fact that
$$
d_i(x,z)+d_i(z,w)\geq d_i(x,w)
$$
if $x,z,w\in\sigma_i(D)$ and $d_i(x,z)<r(x)$.

Suppose $x'\in \sigma_i^{-1}(x)$, $z',z''\in\sigma_i^{-1}(x)$, $w''\in\sigma_i^{-1}(w)$ 
are such that
$$
d_i(x,z)=d[l^{(i)}](x',z')<r(x)
$$
and
$$
d_i(z,w)=d[l^{(i)}](z'',w'').
$$
Then by Lemma 3.2, there exists $x''\in \sigma^{-1}_i(x)$ such that
$$
d[l^{(i)}](x'',z'')=d[l^{(i)}](x',z')=d_i(x,z).
$$
So
$$
d_i(x,z)+d_i(z,w)
=
d[l^{(i)}](x'',z'')+d[l^{(i)}](z'',w'')
\geq
d[l^{(i)}](x'',w'')\geq d_i(x,w).
$$
\end{proof}
\begin{prop}
Given a triangulation $T$ of $S$ and $l\in\mathbb R^{E(T)}_\Delta$, the metric $d[l]$ is compatible with the original topology on $S$.
\end{prop}
\begin{proof}
    Let $\Omega$ be the original topology on $S$. 
     Let $\Omega_d$ be the topology on $S$ induced by $d[l]$. Specifically
    $$
   \Omega_d=\{U\subset S:
   \text{ for all $x\in U$ there exists $\epsilon>0$ such that $y\in U$ if $d[l](x,y)<\epsilon$}\}. 
    $$
    
    We will first prove that $(S,d[l])$ induces a weaker topology, i.e., $\Omega_d\subset\Omega$.  
It suffices to prove that for all $x\in S$ and $r>0$, 
$$
B(x,r):=\{y\in S: 
d[l](x,y)<r\}
$$ 
contains some open set $U\ni x$ in $(S,\Omega)$.
Fix $x\in S$ and let 
$$
K_i
=\bigg\{y'\in D:d[l^{(i)}](x',y')\geq r\text{ for all $x'\in\sigma_i^{-1}(x)$}\bigg\}
$$
for all $i\in I$. Notice that $K_i=D$ if $x\notin\sigma_i(D)$.
Then $K_i$ is compact and
$$
\sigma_i(D)-\sigma_i(K_i)
\subset
\sigma_i\bigg(\big\{y'\in D:d[l^{(i)}](x',y')<r\text{ for some $x'\in\sigma_i^{-1}(x)$}\big\}\bigg)\subset B(x,r)
$$
for all $i\in I$.
It is easy to see that $x$ is contained in the open set
$$
S-\bigcup_{i\in I}\sigma_i(K_i)\subset \bigcup_{i\in I}(\sigma_i(D)-\sigma_i(K_i))\subset B(x,r).
$$
So we proved that $\Omega_d$ is a weaker topology than $\Omega$.

    Then $\Omega=\Omega_d$ since the identity map 
    $$
    Id_{S}:(S,\Omega)\rightarrow(S,\Omega_d)
    $$ 
    is a homeomorphism, as a continuous bijection from a compact space to a Hausdorff space. 
    
    Here we give a detailed direct argument. We only need to show that any closed subset in $(S,\Omega)$ is compact in $(S,d[l])$ and thus is also closed in $(S,d[l])$. Suppose $Y$ is a closed subset in the compact space $(S,\Omega)$, and 
    $$
    Y=\cup_{\lambda\in\Lambda} U_\lambda
    $$ 
    is an open cover of $Y$ in $(S,\Omega_d)$. Then $Y$ is compact in $(S,\Omega)$ and
    $Y=\cup_{\lambda\in\Lambda} U_\lambda$ is also an open cover in the stronger topology $\Omega$. Then there exists a finite subcover of $Y=\cup_{\lambda\in\Lambda}U_\lambda$. So $Y$ is compact in $\Omega_d$.
\end{proof}

\section{Lengths-preserving property}
It is straightforward to verify that $d[l]$ constructed in the previous section is an \emph{intrinsic  metric}, meaning that the distance between two points $x,y$ is the infimum of the lengths of the curves connecting $x,y$. Specifically,
$$
d[l](x,y)=\inf\bigg\{L(\gamma):\gamma\in C([0,1], S),~\gamma(0)=x,~\gamma(1)=y\bigg\}
$$
where $C([0,1],S)$ is the set of the continuous maps from $[0,1]$ to $S$ and
$L(\gamma)=L(\gamma,d)\in[0,\infty]$ measures the length of $\gamma$ as
$$
L(\gamma)=\sup\bigg\{\sum_{i=1}^kd[l](\gamma(t_{i-1}),\gamma(t_i)):0=t_0<t_1<...<t_k=1\bigg\}.
$$
Since we view $d[l]$ as the metric on a surface glued from Euclidean triangles $(D,d[l^{(i)}])$, $d[l]$ should preserve the lengths of the curves on $(D,d[l^{(i)}])$.
Theorem 4.1 says that $d[l]$ is indeed the only intrinsic metric on $S$ preserving these lengths.

\begin{theorem}
(a) For all $i\in I$ and all continuous maps $\gamma:[0,1]\rightarrow D$, 
$$
L(\gamma,d[l^{(i)}])=L(\sigma_i\circ\gamma,d[l]).
$$

(b) Suppose $d$ is an intrinsic metric on $S$. If
for all $i\in I$ and all continuous maps $\gamma:[0,1]\rightarrow D$, 
$$
L(\gamma,d[l^{(i)}])=L(\sigma_i\circ\gamma,d),
$$
then $d=d[l]$.
\end{theorem}

\begin{proof}
(a) 
It is straightforward to verify that
$$
L(\gamma,d[l^{(i)}])\geq L(\sigma_i\circ\gamma,d[l]),
$$
since 
$$
d[l^{(i)}](x,y)
\geq d_i(\sigma_i(x),\sigma_i(y))
\geq d[l](\sigma_i(x),\sigma_i(y))
$$
for all $x,y\in D$.

Now we show
$$
L(\gamma,d[l^{(i)}])\leq L(\sigma_i\circ\gamma,d[l]).
$$
It suffices to show that for all
$0=t_0<t_1<...<t_k=1$,
$$
\sum_{j=1}^kd[l^{(i)}](\gamma(t_{j-1}),\gamma(t_j))
\leq L(\sigma_i\circ\gamma,d[l])
=\sum_{j=1}^k L(\sigma_i\circ\gamma|_{[t_{j-1},t_j]},d[l]).
$$
Comparing the terms on both sides and by  a scaling it suffices to show that for all $i\in I$ and $\alpha\in C([0,1],D)$,
$$
d[l^{(i)}](\alpha(0),\alpha(1))\leq L(\sigma_i\circ\alpha,d[l]).
$$
For all $x\in S$, let
$$
r'(x)=
\min\bigg\{\frac{1}{2}d[l^{(i)}](x',x''):i\in I,~x',x''\in\sigma^{-1}_i(x),~x'\neq x''\bigg\},
$$
and
$$
r''(x)=\min\{r(x),r'(x)\}>0.
$$
Then if $x',z'\in D$ and
$$
d[l^{(i)}](x',z')<r''(x),
$$ 
by Lemma 3.2
there exists $x''\in D$ with $\sigma_i(x'')=\sigma_i(x)$ such that
$$
d[l^{(i)}](x'',z')=d_i(\sigma_i(x'),\sigma_i(z'))<r''(x)
$$ 
and then by the definition of $r'(x)$ we have
$x'=x''$ and
$$
d[l^{(i)}](x',z')=d_i(\sigma_i(x'),\sigma_i(z')).
$$

We will find $0=t_0<...<t_k=1$ such that
$$
d[l^{(i)}](\alpha(t_{j-1}),\alpha(t_j))
<\max\{r''(\sigma_i\circ\alpha(t_{j-1})),r''(\sigma_i\circ\alpha(t_{j}))\}
$$
for all $j=1,...,k$,
and then by the above discussion, Proposition 4.4, and Corollary 4.3,
$$
d[l^{(i)}](\alpha(t_{j-1}),\alpha(t_j))
=d_i(\sigma_i(\alpha(t_{j-1})),\sigma_i(\alpha(t_j)))
=d[l](\sigma_i\circ\alpha(t_{j-1}),\sigma_i\circ\alpha(t_j))
$$
and we are done by the triangle inequality on $(D,d[l^{(i)}])$ and the definition of 
$L(\sigma_i\circ\alpha,d[l])$.

Let $A$ be the set of $t_*\in[0,1]$ such that
there exist
$0=t_0<...<t_k=t_*$ with
$$
d[l^{(i)}](\alpha(t_{j-1}),\alpha(t_j))
<\max\{r''(\sigma_i\circ\alpha(t_{j-1})),r''(\sigma_i\circ\alpha(t_{j}))\}.
$$
If $1\in A$ we are done. Suppose $1\notin A$ and $\bar t=\sup A$. Then it is elementary to verify that
$$
\max\bigg\{1,t_*+\frac{1}{2}r''(\sigma_i\circ\alpha(t_*))\bigg\}\in A,
$$
contradicting with $1\notin A$ or $t_*=\sup A$.

(b)
We first prove that    
$$
d(x,y)\leq d[l](x,y)=
\inf
\bigg\{
\sum_{j=1}^kd_{i_j}(x_{j-1},x_j):x_{j-1},x_j\in\sigma_{i_j}(D),~
x=x_0,~y=x_k
\bigg\}.
$$
Suppose $x_0,...,x_k$ are such that $x_0=x$, $x_k=y$, and for all $j=1,...,k$, $x_{j-1},x_j\in\sigma_{i_j}$ for some $i_j\in I$. Then construct a curve $\gamma$ by connecting $x_{j-1},x_j$ by a ``shortest arc'' for all $j=1,...,k$. Specifically, for all $t\in[0,1/k]$ and $j=1,...,k$, let
$$
\gamma \bigg(\frac{j-1}{k}+t\bigg)=\sigma_j\bigg(ktx_j''+(1-kt)x_{j-1}'\bigg)
$$
where $x_{j-1}''\in \sigma_{i_j}^{-1}(x_{j-1}),x_j'\in\sigma_{i_j}^{-1}(x_j)$ are such that
$$
d_{i_j}(x_{j-1},x_j)=d[l^{(i_j)}](x_{j-1}'',x_j').
$$
Then
$$
d(x,y)\leq L(\gamma,d)
$$
$$
=\sum_{j=1}^k L(\gamma|_{[(j-1)/k,j/k]},d)
$$
$$
=\sum_{j=1}^k L\bigg(\big(t\mapsto (ktx_j''+(1-kt)x_{j-1}')\big)|_{[0,1/k]},~d[l^{(i_j)}] \bigg)
$$
$$
=\sum_{j=1}^k d[l^{(i_j)}](x_{j-1}'',x_j')
$$
$$
=\sum_{j=1}^k d_{i_j}(x_{j-1},x_j).
$$
Notice that the above inequality only uses the fact that $\gamma$ is map from $[0,1]$ to $(S,d)$ and does not require that $\gamma$ is continuous from $[0,1]$ to $(S,d)$.

Now we prove that
$$
d(x,y)\geq d[l](x,y)
$$
for all $x,y\in D$.
Since
$
d(x,y)\leq d[l](x,y)
$
for all $x,y\in D$,
$$
Id:(S,d[l])\rightarrow (S,d)
$$ 
is a continuous bijection from a compact space to a Hausdorff space, so it is a homeomorphism and $(S,d)$ is also compact. 
Let us prove by contradiction and assume
$$
d(x,y)< d[l](x,y)
$$
for some $x,y\in D$.
Since $(S,d)$ is compact, by the Arzel\`a–Ascoli Theorem it is easy to see that there exists a shortest arc in $(S,d)$ connecting $x,y$. 
Speicifically, there exists a continuous map $\gamma:[0,1]\rightarrow S$ such that
$\gamma(0)=x$, $\gamma(1)=y$, and $L(\gamma,d)=d(x,y)$.
As a shortest arc, $\gamma$ is injective. Let 
$$
0=t_0\leq t_1<t_2<....<t_{k-1}\leq t_k=1
$$
be such that 
$$
\{t_1,...,t_{k-1}\}=\{t\in[0,1]:\gamma(t)\in V(T)\}.
$$
Then by the triangle inequality and the additivity of the length of a curve,
$$
d[(\gamma(t_{j-1}),\gamma(t_j))]=L(\gamma|_{[t_{j-1},t_j]})<d[l](\gamma(t_{j-1}),\gamma(t_j))
$$
for some $j\in\{1,...,k\}$.
So without loss of generality, by replacing $\gamma$ with $\gamma|_{[t_{j-1},t_j]}$
we may assume that 
$$
\gamma(0,1)\cap V(T)=\emptyset.
$$
Then by the continuity, we can replace $\gamma$ with $\gamma|_{[\epsilon,1-\epsilon]}$ for sufficiently small $\epsilon>0$
we may assume that
$$
\gamma[0,1]\cap V(T)=\emptyset.
$$
If $\gamma(t)\in\sigma_i(D^o)$ for some $i\in I$ and $t\in[0,1]$, then 

(a) there exists a maximal open interval $A=A_t\subset[0,1]$ such that $t\in A$ and
$$
\gamma(A)\subset\sigma_i(D^o),
$$

(b) $\sigma_i^{-1}\circ\gamma (A)$ is an open straight line since $\gamma$ is locally shortest,

(c) 
$$
L(\gamma|_A,d)=L(\sigma_i^{-1}\circ\gamma|_A,d[l^{(i)}])\geq\epsilon
$$
for some constant $\epsilon=\epsilon(S,T,l,\gamma)>0$ by a standard compactness argument,

(d) the length of the interval $A$ is at least $\delta$ for some $\delta=\delta(S,T,l,\gamma)>0$ by the uniform continuity of $\gamma$,

(e) 
$$
\mathcal A=\{A_t:t\in[0,1], ~\gamma(t)\in\sigma_i(D^o)\text{ for some $i\in I$}\}
$$ 
is finite,

(f) $[0,1]-\sup_{A\in\mathcal A}A$ is a finite disjoint union of points $\{t_1,...,t_{k-1}\}$ and closed nondetenerate intervals $\{B_1,...,B_m\}$.

For all $B_j$, suppose $A_j$ is an open interval adjacent to $B_j$, then it is easy to find a map $\tilde\gamma:A_j\cup B_j\rightarrow D$ such that
$$
\sigma_i\circ\tilde\gamma=\gamma
$$
on $A_j\cup B_j$.
Then $\tilde\gamma$ is not locally shortest, and then neither is $\gamma=\sigma_i\circ\tilde\gamma$. So
$\{B_1,...,B_m\}$ is actually empty. Let $t_0=0,t_k=1$ and we have that
$(t_{j-1},t_j)=\emptyset$ or $(t_{j-1},t_j)\in \mathcal A$.
Then
$$
d(x,y)=L(\gamma,d)
=\sum_{j=1}^k L(\gamma|_{[t_{j-1},t_j]},d)
=\sum_{j=1}^k L(\sigma^{-1}_{i_j}\circ\gamma|_{(t_{j-1},t_j)},d[l^{(i)}])
$$
$$
\geq\sum_{j=1}^k d_{i_j}(\gamma(t_{j-1}),\gamma(t_j))
\geq\sum_{j=1}^k d[l](\gamma(t_{j-1}),\gamma(t_j))
\geq d[l](x,y).
$$
\end{proof}

\section{Teichm\"uller space of polyhedral metrics}
$(S,V)$ is called a \emph{marked surface} if $S$ is a connected closed surface and $V$ is a finite subset of $S$. $T$ is a \emph{triangulation of a marked surface} $(S,V)$ if $T$ is a triangulation of $S$ with $V(T)=V$. By an elementary calculation, 
$$
|E(T)|=3|V|-3\chi(S)
$$
where $|X|$ denotes the number of elements in the set $X$.
From now on, we always assume that $|V|-\chi(S)>0$.

A metric $d$ on $S$ is called a \emph{polyhedral metric of a marked surface} $(S,V)$ if there exists a triangulation $T$ of $(S,V)$ and $l\in\mathbb R^{E(T)}_\Delta$ such that $d=d_T[l]$.
There is an equivalent definition of the polyhedral metrics.

A \emph{standard cone of angle $\theta\in(0,\infty)$} is defined to be the quotient topological space
$$
C_\theta=\{(r,t):r\geq0,t\in\mathbb R/\theta\mathbb Z\}/_{(0,t)\sim (0,t')}
$$
with a Riemannian metric $ds^2=dr^2+r^2dt^2$ on $C^{o}_\theta=C_\theta-\{[(0,0)]\}$. This Riemannian metric $ds^2$ naturally induces an intrinsic metric $d_\theta$ on $C_\theta$.

\begin{definition}
    An intrinsic metric $d$ on a closed surface $(S,V)$ is called \emph{polyhedral} if 
    
    (a) for all $x\in S-V$ there exists an open neighborhood of $x$ isometric to some open domain in the Euclidean plane, and

(b) 
for all $x\in V$ there exists an open neighborhood of $x$ isometric to a neighborhood of $[(0,0)]$ in $(C_\theta,d_\theta)$ with $x$ mapped to $[(0,0)]$ for some $\theta>0$.
\end{definition}
It is intuitively correct and routine to verify that the metric $d[l]$ constructed in Section 3 is a polyhedral metric defined in Definition 5.1. It is also well-known that a polyhedral metric defined in Definition 5.1 could always be represented by $d[l]$ for some triangulation $T$ on $(S,V)$ and $l\in\mathbb R^{E(T)}_\Delta$. For example, \cite{bobenko2007discrete} describes a construction of such representation. In a summary,
\begin{theorem}
Given $(S,V)$,
$$
\{\text{polyhedral metrics on $(S,V)$}\}
$$
$$
=
\{d[l]:\text{$l\in\mathbb R^{E(T)}_\Delta$ where $T$ is a triangulation of $(S,V)$}\}
$$
    
\end{theorem}

In many cases, we do not care about the differences between two isometric polyhedral metrics on a marked surface $(S,V)$. So it would be convenient and useful to reduce the space of polyhedral metrics to a finitely dimensional space. 
\begin{definition}
Given a polyhedral metric $d$ on $(S,V)$, denote
$$
[d]=\{f^*d: f\in \mathcal H_0(S,V)
\}
$$ 
where

(a) $f^*d$ is the pullback metric on $S$ defined as 
$$
f^*d(x,y)=d(f(x),f(y)),
$$ 
and

(b) 
$$
\mathcal H_0(S,V)=\{f:S\rightarrow S\big|\text{ $f$ is a homeomorphism and $f|_{S-V}$ is isotopic to $Id|_{S-V}$}\}.
$$
The resulted quotient set 
$$
\{[d]:\text{$d$ is a polyhedral metric on $(S,V)$}\}
$$
is denoted as $T_{PL}(S,V)$ and called the \emph{Teichm\"uller space} of the polyhedral metrics on $(S,V)$. 
\end{definition}

\subsection{Manifold structure of $T_{PL}$ and the edge lengths coordinates}
By a standard argument, one can show that
\begin{prop}
    Given a triangulation $T$ of $(S,V)$, $l\mapsto \big[d_T[l]\big]$ is injective from $\mathbb R^{E(T)}_\Delta$ to $T_{PL}(S,V)$.
\end{prop}
By relating to the Teichm\"uller space of hyperbolic metrics, we will briefly show that

(a) $T_{PL}(S,V)$ is naturally homeomorphic to $\mathbb R^{3|V|-3\chi(S)}$ when $|V|-\chi(S)>0$, and

(b) $l\mapsto \big[d_T[l]\big]$ induces an atlas on $T_{PL}(S,V)$.

Given a marked surface $(S,V)$ and a complete hyperbolic metrics $d$ of finite area on $S-V$, denote
$$
[d]=\{f^*d:f\in\mathcal H_0(S-V)\}
$$
where 
$
\mathcal H_0(S-V)
$
is the collection of homeomorphisms from $S-V$ to $S-V$ that are isotopic to $id|_{S-V}$.
We denote 
$$
T(S-V)=\{[d]:
\text{$d$ is a complete hyperbolic metrics on $(S-V)$ of finite area}
\}
$$ 
and recall that $T(S-V)$ is called the \emph{Teichm\"uller space of the complete hyperbolic metrics of finite area} on $S-V$.
It is well-known that $T(S-V)$ is homeomorphic to $\mathbb R^{2|V|-3\chi(S)}$.

Given a polyhedral metric $d$ on $(S,V)$, denote $d_h$ as the unique complete hyperbolic metric on $S-V$ that is conformal to $d$ on $S-V$. By estimating the extremal lengths, it is routine to verify that all the topological ends of $(S-V,d_h)$ are cusp ends. So $d_h$ is a complete hyperbolic metric on $S-V$ of finite area and $[d_h]\in T(S-V)$.

Given a polyhedral metric $d$ on $(S,V)$, denote $A(d)$ as the area of $(S,d)$ and $\theta(d)\in\mathbb R^{V}_{>0}$ as the cone angles on $V$. By the Gauss–Bonnet theorem 
$$
\sum_{v\in V}(2\pi-\theta_v(d))=2\pi\chi(S).
$$

Troyanov (1986)
proved that  every compact Riemann surface carries a unique, up to homothety, Euclidean (flat) conformal metric with prescribed conical singularities of given angles, provided the Gauss-Bonnet relation is satisfied. In other words, he proved that
\begin{theorem}[Troyanov]
$d\mapsto (d_h,\theta(d),A(d))$ is a bijection from $\tilde T_{PL}(S,V)$ to 
$$
\tilde T(S-V)\times GB\times\mathbb R_{>0}
$$
where 

(a) $\tilde T_{PL}(S,V)$ is the set of the polyhedral metric on $(S,V)$,

(b) $\tilde T(S-V)$ is the set of the complete hyperbolic metric of finite area on $S-V$, and

(c) 
$$
GB=\{\theta\in\mathbb R^V_{>0}:
\sum_{v\in V}(2\pi-\theta_i(d))=2\pi\chi(S)
\}.
$$
\end{theorem}
As a consequence, we have that
\begin{corollary}
    $[d]\mapsto ([d_h],\theta(d),A(d))$ is a bijection from $ T_{PL}(S,V)$ to 
$$
 T(S-V)\times GB\times\mathbb R_{>0}.
$$
\end{corollary}
This bijection naturally assign $T_{PL}(S,V)$ a topology that is homeomorphic to $\mathbb R^{3|V|-3\chi(S)}$.
We will see that for any triangulation $T$ of $(S,V)$, $l\mapsto d_T[l]$ is an embedding from $\mathbb R^{E(T)}_\Delta$ to $T_{PL}(S,V)$. So $l\mapsto d_T[l]$ naturally induces an atlas on $T_{PL}(S,V)$. To see that $l\mapsto d_T[l]$ is an embedding, by Proposition 5.4 and the open mapping theorem it suffices to show that $l\mapsto (d_T[l])_h$
is continuous. Fix $\bar l\in\mathbb R^{E(T)}_\Delta$ and then
$$
Id|_{S}:(S,d_T[\bar l])\rightarrow (S,d_T[l])
$$
is $K_l$-quasiconformal where
$
K_l\rightarrow1
$
as $l\rightarrow\bar l$.
So
the Teichm\"uller distance between $[(d_T[l])_h]$ and $[(d_T[\bar l])_h]$ goes to $0$ as $l\rightarrow\bar l$.

\subsection{Another approach using discrete conformal geometry}
Gu et al. \cite{gu2018discrete} introduced the discrete uniformization theorem in the Teichm\"uller space of polyhedral metrics on marked surfaces. This discrete uniformization theorem gives another $\mathbb R^n$-parameterization of $T_{PL}$, by a similar argument.
Related work on discrete conformality and geodesic triangulations could be found in 
\cite
{wu2022surface,
luo2023convergence,
wu2021fractional,
luo2023deformation,
dai2023rigidity,
wu2023computing,
izmestiev2023prescribed,
luo2021deformation2,
luo2019koebe,
wu2020convergence,
luo2021deformation,
luo2022discrete,
wu2015rigidity,
sun2015discrete,
gu2018discrete,
gu2018discrete2,
gu2019convergence
}.

\bibliography{sreference}
\bibliographystyle{amsplain}

\end{document}